\documentclass[12pt]{amsart}

\usepackage{multirow}
\usepackage{hhline}

\usepackage{mathtools}
\usepackage{times}
\usepackage[T1]{fontenc}
\usepackage{mathrsfs}
\usepackage{latexsym}
\usepackage[dvips]{graphics}
\usepackage[titletoc, title]{appendix}
\usepackage{epsfig}
\usepackage{amssymb}
\usepackage{amsmath,amsfonts,amsthm,amssymb,amscd}
\usepackage{color}
\usepackage{hyperref}
\usepackage{amsmath}

\usepackage{color}
\usepackage{breakurl}

\usepackage{comment}
\newcommand{\bburl}[1]{\textcolor{blue}{\url{#1}}}

\newcommand{\be}{\begin{equation}}
\newcommand{\ee}{\end{equation}}

\DeclareMathOperator{\sgn}{sgn}
\newtheorem{thm}{Theorem}[section]

\newtheorem{cor}[thm]{Corollary}
\newtheorem{claim}[thm]{Claim}
\newtheorem{lem}[thm]{Lemma}
\newtheorem{prop}[thm]{Proposition}

\newtheorem{defi}[thm]{Definition}
\newtheorem{rek}[thm]{Remark}

\DeclareMathOperator{\supp}{supp}
\DeclareMathOperator{\spann}{span}

\numberwithin{equation}{section}

\begin{document}

\title[Schreier Families \& $\mathcal{F}$-greedy-type bases]{Schreier Families and $\mathcal{F}$-(almost) greedy bases}

\author{Kevin Beanland}

\email{\textcolor{blue}{\href{mailto:beanlandk@wlu.edu}{beanlandk@wlu.edu}}}
\address{Department of Mathematics, Washington and Lee University, Lexington, VA 24450, USA.}

\author{H\`ung Vi\d{\^e}t Chu}

\email{\textcolor{blue}{\href{mailto:hungchu2@illinois.edu}{hungchu2@illinois.edu}}}
\address{Department of Mathematics, University of Illinois at Urbana-Champaign, Urbana, IL 61820, USA}

\begin{abstract} 
Let $\mathcal{F}$ be a hereditary collection of finite subsets of $\mathbb{N}$. In this paper, we introduce and characterize $\mathcal{F}$-(almost) greedy bases.  Given such a  family $\mathcal{F}$, a basis $(e_n)_n$ for a Banach space $X$ is called $\mathcal{F}$-greedy if there is a constant $C\geqslant 1$ such that for each $x\in X$, $m \in \mathbb{N}$, and $G_m(x)$, we have
$$\|x - G_m(x)\|\ \leqslant\ C \inf\left\{\left\|x-\sum_{n\in A}a_ne_n\right\|\,:\, |A|\leqslant m, A\in \mathcal{F}, (a_n)\subset \mathbb{K}\right\}.$$
Here $G_m(x)$ is a greedy sum of $x$ of order $m$, and $\mathbb{K}$ is the scalar field. From the definition, any $\mathcal{F}$-greedy basis is quasi-greedy and so, the notion of being $\mathcal{F}$-greedy lies between being greedy and being quasi-greedy. We characterize $\mathcal{F}$-greedy bases as being $\mathcal{F}$-unconditional, $\mathcal{F}$-disjoint democratic, and quasi-greedy, thus generalizing the well-known characterization of greedy bases by Konyagin and Temlyakov. We also prove a similar characterization for $\mathcal{F}$-almost greedy bases.

Furthermore, we provide several examples of bases that are nontrivially $\mathcal{F}$-greedy. For a countable ordinal $\alpha$, we consider the case  $\mathcal{F}=\mathcal{S}_\alpha$, where $\mathcal{S}_\alpha$ is the Schreier family of order $\alpha$. We show that for each $\alpha$, there is a basis that is $\mathcal{S}_{\alpha}$-greedy but is not $\mathcal{S}_{\alpha+1}$-greedy. In other words, we prove that none of the following implications can be reversed: for two countable ordinals $\alpha < \beta$, 
$$\mbox{quasi-greedy}\ \Longleftarrow\ \mathcal{S}_\alpha\mbox{-greedy}\ \Longleftarrow\ \mathcal{S}_\beta\mbox{-greedy}\ \Longleftarrow\ \mbox{greedy}.$$

\end{abstract}



\thanks{The second author acknowledges the summer funding from the Department of Mathematics at the University of Illinois at Urbana-Champaign.}

\maketitle

\tableofcontents

\section{Introduction}

A (semi-normalized) \textit{basis} in a Banach space $X$ over the field $\mathbb{K}$ is a countable collection $(e_n)_n$ such that 
\begin{enumerate}
    \item[i)] $\overline{\spann\{e_n: n\in\mathbb{N}\}} = X$,
    \item[ii)] there exists a unique sequence $(e_n^*)_n\subset X^*$ such that $e_i^*(e_j) = \delta_{i, j}$ for all $i, j\in\mathbb{N}$, and
    \item[iii)] there exist $c_1, c_2 > 0$ such that 
    $$0 \ <\ c_1 := \inf_n\{\|e_n\|, \|e_n^*\|\}\ \leqslant\ \sup_n\{\|e_n\|, \|e_n^*\|\} \ =:\ c_2 \ <\ \infty.$$
\end{enumerate}

In 1999, Konyagin and Temlyakov \cite{KT1} introduced the Thresholding Greedy Algorithm (TGA), which picks the largest coefficients (in modulus) for the approximation. In particular, for each $x\in X$ and $m\in\mathbb{N}$, a set $\Lambda_m(x)$ is a \textit{greedy set} of order $m$ if $|\Lambda_m(x)| = m$ and $\min_{n\in\Lambda_m(x)}|e_n^*(x)|\geqslant \max_{n\notin \Lambda_m(x)}|e_n^*(x)|$. A \textit{greedy operator} $G_m: X \to X$ is defined as 
$$G_m(x) \ =\ \sum_{n\in \Lambda_m(x)}e_n^*(x)e_n,\mbox{ for some }\Lambda_m(x).$$
Note that $\Lambda_m(x)$ (and thus, $G_m(x)$) may not be unique and $G_m$ is not even linear. 
The TGA is a sequence of greedy operators $(G_m)_{m=1}^\infty$ that gives the corresponding sequence of approximants $(G_m(x))_{m=1}^\infty$  for each $x\in X$.

A basis $(e_n)_n$ for a Banach space $X$ is called \textit{greedy} if there is a $C\geqslant 1$ such that for all $x\in X,m\in \mathbb{N}$, and $G_m$,
$$\|x - G_m(x)\| \ \leqslant\ \  C \inf\left\{\left\|x-\sum_{n\in A}a_ne_n\right\|\,:\, |A|\leqslant m, (a_n)\subset \mathbb{K}\right\}.$$
A basis is called quasi-greedy \cite{KT1} if there is a $C\geqslant 1$ so that for all $x\in X, m\in \mathbb{N}$, and $G_m$, we have $\|G_m(x)\| \leqslant  C\|x\|$. The smallest such $C$ is denoted by $\mathbf C_w$, called the \textit{quasi-greedy constant}. 
Also for quasi-greedy bases, let $\mathbf C_\ell$, called the \textit{suppression quasi-greedy constant}, be the smallest constant such that 
$$\|x-G_m(x)\| \ \leqslant\ \mathbf C_\ell\|x\|,\forall x\in X, \forall m\in\mathbb{N}, \forall G_m.$$
There are many examples of quasi-greedy bases that are not greedy (see \cite[Example 10.2.9]{AK}), and there has been research on the existence of greedy bases for certain classical spaces (\cite{DFOS, G}). 

In this paper, we introduce and study the notion of what we call $\mathcal{F}$-greedy bases which interpolate between greedy bases and quasi-greedy bases. Recall that a collection $\mathcal{F}$ of finite subsets of $\mathbb{N}$ is said to be \textit{hereditary} if $F\in \mathcal{F}$ and $G \subset F$ imply $G \in \mathcal{F}$.

\begin{defi}\normalfont
Let $\mathcal{F}$ be a hereditary collection of finite subsets of $\mathbb{N}$. A basis $(e_n)_n$ is $\mathcal{F}$-greedy if there exists a constant $C\geqslant 1$ such that for all $x\in X,m\in \mathbb{N}$, and $G_m$,
$$\|x - G_m(x)\|\ \leqslant \ C\sigma_m^{\mathcal{F}}(x),$$
where $$\sigma_m^{\mathcal{F}}(x)\ := \inf\left\{\left\|x-\sum_{n\in A}a_ne_n\right\|\,:\, |A|\leqslant m, A\in \mathcal{F}, (a_n)\subset \mathbb{K}\right\}.$$
The least constant $C$ is denoted by $\mathbf{C}_g^{\mathcal{F}}$.
\end{defi}

\begin{rek}In the case the $\mathcal{F}=\mathcal{P}(\mathbb{N})$,  $\mathcal{F}$-greedy corresponds to greedy and when $\mathcal{F}=\{\emptyset\}$, $\mathcal{F}$-greedy corresponds to quasi-greedy.
\end{rek}

The first order of business is to generalize the theorem of Konyagin and Temlyakov, which characterizes greedy bases as being  unconditional and democratic. To do so, we introduce the  definitions of $\mathcal{F}$- unconditionality and $\mathcal{F}$-democracy. For various families $\mathcal{F}$, the notion of $\mathcal{F}$-unconditionality has appeared numerous times in the literature, most notably in Odell's result \cite{Od}, which states that every normalized weakly null sequence in a Banach space has a subsequence that is Schreier-unconditional. Also see \cite{AG, AGR, AMT} for other notion of unconditionality for weakly null sequences. 

For a basis $(e_n)_n$ of a Banach space $X$ and a finite set $A \subset \mathbb{N}$, let $P_A: X \to X$ be defined by $P_A(\sum_i e^*_i(x) e_i)=\sum_{i\in A} e^*_i(x) e_i$.

\begin{defi}\normalfont
A basis $(e_n)$ of a Banach space $X$ is $\mathcal{F}$-unconditional if there exists a constant $C\geqslant 1$ such that for each $x \in X$ and $A \in \mathcal{F}$, we have
$$\|x- P_A(x)\|\ \leqslant\ C\|x\|.$$
The least constant $C$ is denoted by $\mathbf K_s^{\mathcal{F}}$. We say that $(e_n)$ is $\mathbf K_s^{\mathcal{F}}$-$\mathcal{F}$-suppression unconditional.
\end{defi}

As far as we know, the following  natural definition has not appeared in the literature before. 

\begin{defi}\label{demo}\normalfont
A basis $(e_n)$ is $\mathcal{F}$-disjoint democratic ($\mathcal{F}$-disjoint superdemocratic, respectively) if there exists a constant $C\geqslant 1$ such that 
$$\left\|\sum_{i\in A}e_i\right\|\ \leqslant \ C\left\|\sum_{i\in B}e_i\right\|,\mbox{ }\left(\left\|\sum_{i\in A}\varepsilon_i e_i\right\|\ \leqslant\ C\left\|\sum_{i\in B} \delta_i e_i\right\|,\mbox{ respectively}\right),$$
for all finite sets $A, B\subset\mathbb{N}$ with $A\in \mathcal{F}$, $|A|\leqslant |B|, A\cap B = \emptyset$ and signs $(\varepsilon_i), (\delta_i)$. The least constant $C$ is denoted by $\mathbf C^{\mathcal{F}}_{d,\sqcup}$ ($\mathbf C^{\mathcal{F}}_{sd,\sqcup}$, respectively). When $\mathcal{F} = \mathcal{P}(\mathbb{N})$, we say that $(e_n)$ is (super)democratic. 
\end{defi}

One of our main results is the following generalization of the Konyagin-Temlyakov Theorem \cite{KT1}.

\begin{thm}\label{m1'}
A basis $(e_n)$ in a Banach space $X$ is $\mathcal{F}$-greedy if and only if it is quasi-greedy, $\mathcal{F}$-unconditional, and $\mathcal{F}$-disjoint democratic.
\end{thm}

 We also present another characterization regarding $\mathcal{F}$-almost greedy bases.

\begin{defi}\normalfont
A basis $(e_n)$ is $\mathcal{F}$-almost greedy if there exists a constant $C\geqslant 1$ such that for all $x \in X, m\in \mathbb{N}$, and $G_m$, we have
$$\|x-G_m(x)\| \ \leqslant\ C\inf\{\|x-P_A(x)\|\,:\, |A|\leqslant m, A\in \mathcal{F}\}.$$
The least constant $C$ is denoted by $\mathbf{C}_a^{\mathcal{F}}$.
\end{defi}

The next theorem generalizes \cite[Theorem 3.3]{DKKT}. 

\begin{thm}
A basis $(e_n)$  is $\mathcal{F}$-almost greedy if and only if it is quasi-greedy and $\mathcal{F}$-disjoint democratic. 
\end{thm}

The second set of results in this paper  focuses on the well-known Schreier families $(\mathcal{S}_\alpha)_{n=1}^\infty$ (for each countable ordinal $\alpha$) introduced by Alspach and Argyros \cite{AA2}. The sequence of countable ordinals is
$$0, 1, \ldots, n, \ldots, \omega, \omega + 1, \ldots, 2\omega,\ldots, $$
We recall the definition of $\mathcal{S}_\alpha$. For two sets $A, B\subset\mathbb{N}$, we write
$A < B$ to mean that $a < b$ for all $a\in A, b\in B$. It holds vacuously that $\emptyset < A$ and $\emptyset > A$. Also, $n < A$ for a number $n$ means $\{n\} < A$. 
Let $\mathcal{S}_0$ be the set of singletons and the empty set. Supposing that $\mathcal{S}_{\alpha}$ has be defined for some ordinal $\alpha \geqslant 0$, we define
\begin{equation*}\mathcal{S}_{\alpha+1}\ =\ \{\cup^{m}_{i=1} E_i:m\leqslant E_1<E_2<\cdots<E_m \mbox{ and }  E_i\in \mathcal{S}_\alpha, \forall 1\leqslant i\leqslant m\}.\end{equation*}
If $\alpha$ is a limit ordinal, then fix $\alpha_{m}+1\nearrow \alpha$ with $\mathcal{S}_{\alpha_{m}}\subset \mathcal{S}_{\alpha_{m+1}}$ for all $m\geqslant 1$ and define
\begin{equation*}
    \mathcal{S}_{\alpha} \ =\ \{E\subset \mathbb{N}\,:\, \mbox{ for some }m\geqslant 1, m\leqslant E\in \mathcal{S}_{\alpha_m+1}\}.
\end{equation*}

The following proposition is well-known, but we include its proof for completion.

\begin{prop}\label{k2}
Let $\alpha < \beta$ be two countable ordinals. There exists $N\in \mathbb{N}$ such that 
$$E\backslash \{1, \ldots, N-1\}\in \mathcal{S}_{\beta}, \forall E\in \mathcal{S}_\alpha.$$
\end{prop}

\begin{proof}
Fix two ordinals $\alpha < \beta$. 
We prove by induction. Base cases: if $\beta = 0$, there is nothing to prove. If $\beta = 1$, then $\alpha = 0$. Clearly, $\mathcal{S}_0\subset \mathcal{S}_1$. Inductive hypothesis: suppose that the proposition holds for all $\eta < \beta$. If $\beta$ is a successor ordinal, then write $\beta = \gamma + 1$. Since $\alpha < \beta$, we have $\alpha\leqslant \gamma$. By the inductive hypothesis, there exists $N\in\mathbb{N}$ such that 
$$E\backslash \{1,\ldots, N-1\} \in \mathcal{S}_\gamma, \forall E\in \mathcal{S}_\alpha.$$
By definition, $\mathcal{S}_\gamma\subset\mathcal{S}_\beta$. Hence, 
$$E\backslash \{1, \ldots, N-1\}\in \mathcal{S}_\beta, \forall E\in\mathcal{S}_\alpha.$$
If $\beta$ is a limit ordinal, then let $\beta_m\nearrow \beta$. There exists $M\in\mathbb{N}$ such that $\beta_M\geqslant \alpha$. By the inductive hypothesis, there exists $N_1\in\mathbb{N}$ such that
$$E\backslash \{1, \ldots, N_1-1\}\in \mathcal{S}_{\beta_M}, \forall E\in \mathcal{S}_{\alpha}.$$
By definition,
$$E\backslash \{1, \ldots, M-1\}\in \mathcal{S}_{\beta}, \forall E\in \mathcal{S}_{\beta_M}.$$
Therefore,
$$E\backslash \{1, \ldots, \max\{N_1, M\}-1\}\in \mathcal{S}_\beta, \forall E\in \mathcal{S}_{\alpha}.$$
This completes our proof. 
\end{proof}

We have the following corollary, which is proved in Section \ref{pc}.

\begin{cor}\label{m30'}
For two countable ordinals $\alpha < \beta$, an $\mathcal{S}_\beta$-greedy basis is $\mathcal{S}_\alpha$-greedy. 
\end{cor}

Each Schreier family $\mathcal{S}_\alpha$ is obviously hereditary and are moreover spreading and compact (see \cite[pp. 1049 and 1051]{AGR}).
We shall show that each of the following implications cannot be reversed: for two countable ordinals $\alpha < \beta$, 
$$\mbox{quasi-greedy}\ \Longleftarrow\ \mathcal{S}_\alpha\mbox{-greedy}\ \Longleftarrow\ \mathcal{S}_\beta\mbox{-greedy}\ \Longleftarrow\ \mbox{greedy}.$$ 
We, thereby, study the greedy counterpart of the notion of $\mathcal{S}_\alpha$-unconditionality. 

\begin{thm}\label{m30}
For two countable ordinals $\alpha < \beta$, there exists a Banach space $X$ with an $\mathcal{S}_\alpha$-greedy basis that is not $\mathcal{S}_{\beta}$-greedy. 
\end{thm}

\begin{thm}\label{m31}
Fix a countable ordinal $\alpha$. 
\begin{enumerate}
    \item A basis is greedy if and only if it is $C$-$\mathcal{S}_{\alpha+m}$-greedy for all $m\in\mathbb{N}$ and some uniform $C\geqslant 1$.
    \item There exists a basis that is $\mathcal{S}_{\alpha+m}$-greedy (with different constants) for all $m\in \mathbb{N}$ but is not greedy. 
\end{enumerate}
\end{thm}

\section{Characterizations of $\mathcal{F}$-greedy bases}

In this section, we prove Theorem \ref{m1'} and other characterizations of $\mathcal{F}$-greedy bases. Throughout, $\mathcal{F}$ will be a hereditary family of finite subsets of $\mathbb{N}$. We first need to define Property (A, $\mathcal{F}$), inspired by the classical Property (A) introduced by Albiac and Wojtaszczyk in \cite{AW}. Write $\sqcup_{i\in I} A_i$, for some index set $I$ and sets $(A_i)_{i\in I}$, to mean that the $A_i$'s are pairwise disjoint. Define $1_A = \sum_{n\in A}e_n\mbox{ and }1_{\varepsilon A} = \sum_{n\in A}\varepsilon_n e_n$, for some signs $(\varepsilon) = (\varepsilon_n)_n\in \mathbb{K}^{\mathbb{N}}$.

\begin{defi}\label{dAF}\normalfont
A basis $(e_n)$ is said to have Property (A, $\mathcal{F}$) if there exists a constant $C\geqslant 1$ such that
$$\left\|x+\sum_{i \in A} \varepsilon_i e_i\right\|\ \leqslant\   C\left\|x + \sum_{n\in B}b_n e_n\right\|,$$
for all $x\in X$ with $\|x\|_\infty\leqslant 1$, for all finite sets $A, B\subset\mathbb{N}$ with $|A|\leqslant |B|$, $A\in \mathcal{F}$, $A\sqcup B\sqcup \supp(x)$, and for all signs $(\varepsilon_i)$ and $|b_n|\geqslant 1$. The least constant $C$ is denoted by $\mathbf C^{\mathcal{F}}_b$.
\end{defi}

\begin{prop}\label{p1}
A basis $(e_n)$ has $\mathbf C^{\mathcal{F}}_{b}$-Property (A, $\mathcal{F}$) if and only if 
\begin{equation}\label{e1}\|x\|\ \leqslant\ \mathbf C^{\mathcal{F}}_b \left\|x-P_A(x) + \sum_{n\in B}b_n e_n\right\|,\end{equation}
for all
$x\in X$ with $\|x\|_\infty\leqslant 1$, for all finite sets $A, B\subset\mathbb{N}$ with $|A|\leqslant |B|$, $A\in \mathcal{F}$, $B\cap (A\cup \supp(x)) = \emptyset$, and $|b_n|\geqslant 1$.
\end{prop}

\begin{proof}
Assume \eqref{e1}. Let $x, A, B, (\varepsilon), (b_n)_{n\in B}$ be as in Definition \ref{dAF}. Let $y = x + 1_{\varepsilon A}$. By \eqref{e1}, 
$$\|x+1_{\varepsilon A}\|\ =\ \|y\|\ \leqslant\ \mathbf C^{\mathcal{F}}_b\left\|y- P_A(y) + \sum_{n\in B}b_n e_n\right\|\ =\ \mathbf C^{\mathcal{F}}_b \left\|x+\sum_{n\in B}b_ne_n\right\|.$$

Conversely, assume that $(e_n)$ has $\mathbf C^{\mathcal{F}}_b$-Property (A, $\mathcal{F}$). Let $x, A, B, (b_n)_{n\in B}$ be as in \eqref{e1}. We have
\begin{align*}
    \|x\|\ =\ \left\|x-P_A(x) + \sum_{n\in A}e_n^*(x)e_n\right\|&\ \leqslant\ \sup_{(\delta)}\left\|x-P_A(x) + 1_{\delta A}\right\|\mbox{ by norm convexity}\\
    &\ \leqslant\ \mathbf C^{\mathcal{F}}_b\left\|x - P_A(x) + \sum_{n\in B}b_ne_n\right\|,
\end{align*}
where the last inequality is due to Property (A, $\mathcal{F}$).
\end{proof}

\begin{thm}\label{m1}
Let $(e_n)$ be a basis for a Banach space $X$.
\begin{enumerate}
\item The basis $(e_n)$ is $\mathbf C^{\mathcal{F}}_g$-$\mathcal{F}$-greedy, then $(e_n)$ is $\mathbf C^{\mathcal{F}}_g$-$\mathcal{F}$-suppression unconditional and has $\mathbf C^{\mathcal{F}}_g$-Property (A, $\mathcal{F}$).
\item The basis $(e_n)$ is $\mathbf K^{\mathcal{F}}_s$-$\mathcal{F}$-suppression unconditional and has $\mathbf C^{\mathcal{F}}_b$-Property (A, $\mathcal{F}$), then $(e_n)$ is $\mathbf K^{\mathcal{F}}_s\mathbf C^{\mathcal{F}}_b$-$\mathcal{F}$-greedy. 
\end{enumerate}
\end{thm}

\begin{proof}
(1) Assume that $(e_n)$ is $\mathbf C^{\mathcal{F}}_g$-$\mathcal{F}$-greedy. We shall show that $(e_n)$ is $\mathcal{F}$-unconditional. Choose $x\in X$ and a finite set $B\in \mathcal{F}$. Set
$$y\ :=\ \sum_{n\in B}(e_n^*(x)+\alpha)e_n + \sum_{n\notin B} e_n^*(x)e_n,$$
where $\alpha$ is sufficiently large such that $B$ is a greedy set of $y$. Then 
$$\|x-P_B(x)\|\ =\ \|y-P_B(y)\|\ \leqslant\ \mathbf C_g^{\mathcal{F}}\sigma^{\mathcal{F}}_{|B|}(y)\ \leqslant\ \mathbf C_g^{\mathcal{F}}\|y-\alpha 1_B\|\ =\ \mathbf C_g^{\mathcal{F}}\|x\|.$$
Hence, $(e_n)$ is $\mathbf C^{\mathcal{F}}_g$-$\mathcal{F}$-suppression unconditional.

Next, we prove Property (A, $\mathcal{F}$). Choose $x, A, B, (\varepsilon_i), (b_n)_{n\in B}$ as in Definition \ref{dAF}. Set 
$y:= x + 1_{\varepsilon A} + \sum_{n\in B}b_ne_n$. Since $B$ is a greedy set of $y$, we have
$$\|x+1_{\varepsilon A}\|\ =\ \|y-P_B(y)\|\ \leqslant\ \mathbf C_g^{\mathcal{F}}\sigma^{\mathcal{F}}_{|B|}(y)\ \leqslant\ \mathbf C_g^{\mathcal{F}}\|y-P_A(y)\|\ =\ \mathbf C_g^{\mathcal{F}}\left\|x+\sum_{n\in B}b_ne_n\right\|.$$
Therefore, $(e_n)$ has $\mathbf C_g^{\mathcal{F}}$-Property (A, $\mathcal{F}$). 

(2)  Assume that $(e_n)$ is $\mathbf K^{\mathcal{F}}_s$-$\mathcal{F}$-unconditional and has $\mathbf C^{\mathcal{F}}_b$-Property (A, $\mathcal{F}$). Let $x\in X$ with a greedy set $A$. Choose $B\in \mathcal{F}$ with $|B|\leqslant |A|$ and choose $(b_n)_{n\in B}\subset \mathbb{K}$. If $A\backslash B = \emptyset$, then $A = B$ and we have
\begin{align*}\|x-P_A(x)\|\ =\ \|x-P_B(x)\|&\ \leqslant\ \mathbf K^{\mathcal{F}}_s\left\|x-P_B(x) + \sum_{n\in B}(e_n^*(x)-b_n)e_n\right\|\\
&\ =\ \mathbf K^{\mathcal{F}}_s\left\|x - \sum_{n\in B}b_ne_n\right\|.
\end{align*}
Assume that $A\backslash B\neq \emptyset$.
Note that $B\backslash A\in \mathcal{F}$ as $\mathcal{F}$ is hereditary and $\min_{n\in A\backslash B}|e_n^*(x)|\geqslant \|x-P_A(x)\|_\infty$.
By Proposition \ref{p1}, we have
\begin{align*}
    \|x-P_A(x)\|&\ \leqslant\ \mathbf C^{\mathcal{F}}_b\|(x-P_A(x)) - P_{B\backslash A}(x) + P_{A\backslash B}(x)\|\\
    &\ =\ \mathbf C^{\mathcal{F}}_b\|x-P_B(x)\|\\
    &\ \leqslant\ \mathbf C^{\mathcal{F}}_b\mathbf K^{\mathcal{F}}_s\left\|x-P_B(x) + \sum_{n\in B}(e_n^*(x)-b_n)e_n\right\|\\
    &\ =\ \mathbf C^{\mathcal{F}}_b\mathbf K^{\mathcal{F}}_s\left\|x-\sum_{n\in B}b_ne_n\right\|.
\end{align*}
Since $B$ and $(b_n)$ are arbitrary, we know that $(e_n)$ is $\mathbf C^{\mathcal{F}}_b\mathbf K^{\mathcal{F}}_s$-$\mathcal{F}$-greedy. 
\end{proof}

We have the following immediate corollary.

\begin{cor}
A basis $(e_n)$ is $1$-$\mathcal{F}$-greedy if and only if it is $1$-$\mathcal{F}$-unconditional and has $1$-Property (A, $\mathcal{F}$).
\end{cor}

The next proposition connects Property $(A,\mathcal{F})$ and $\mathcal{F}$-disjoint democracy.

\begin{prop}\label{p20}
Let $(e_n)$ be a quasi-greedy basis. Then $(e_n)$ has Property (A, $\mathcal{F}$) if and only if $(e_n)$ is $\mathcal{F}$-disjoint democratic. 
\end{prop}

The proof of Proposition \ref{p20} uses the following results which can be found in \cite{W} and  \cite[Lemma 2.5]{BBG}. 

\begin{lem}\label{bto}
Let $(e_n)$ be a $\mathbf C_\ell$-suppression quasi-greedy basis. The following hold 
\begin{enumerate}
    \item For any finite set $A\subset\mathbb{N}$ and sign $(\varepsilon_n)_n$, we have
$$\frac{1}{2\mathbf C_\ell}\left\|\sum_{n\in A}e_n\right\|\ \leqslant\ \left\|\sum_{n\in A}\varepsilon_ne_n\right\| \ \leqslant\ 2\mathbf C_\ell\left\|\sum_{n\in A}e_n\right\|.$$
\item For all $\alpha >0$ and $x \in X$,
$$\left\|\sum_{n \in \Gamma_\alpha(x)} \alpha \sgn(e_n^*(x))e_n + \sum_{n \not \in \Gamma_\alpha(x)} e_n^*(x)e_n\right\|\ \leqslant\ \mathbf{C}_\ell \|x\|,$$ 
where $\Gamma_\alpha (x) =\{n : |e^*_n(x)|>\alpha\}$.
\end{enumerate}
\end{lem}

\begin{proof}[Proof of Proposition \ref{p20}]
It is obvious that Property (A, $\mathcal{F}$) implies $\mathcal{F}$-disjoint democracy. Let us assume that $(e_n)$ is $\mathbf C^{\mathcal{F}}_{d, \sqcup}$-$\mathcal{F}$-disjoint democratic and is $\mathbf C_\ell$-suppression quasi-greedy (or $\mathbf C_w$-quasi-greedy). Let $x, A, B, (b_n), (\varepsilon_i)$ be as in Definition \ref{dAF}.  Since $B$ is a greedy set of $x+\sum_{n\in B}b_n e_n$, we have
\begin{align*}
    \left\|x+\sum_{n\in B}b_n e_n\right\|\ \geqslant\ \frac{1}{\mathbf C_w}\left\|\sum_{n\in B}b_n e_n\right\|&\ \geqslant\ \frac{1}{\mathbf C_w\mathbf C_\ell}\left\|\sum_{n\in B}\sgn(b_n)e_n\right\|\mbox{ by Lemma \ref{bto}}\\
    &\ \geqslant\ \frac{1}{2\mathbf C_w\mathbf C^2_\ell}\|1_B\|\mbox{ by Lemma \ref{bto}}\\
    &\ \geqslant\ \frac{1}{2\mathbf C_w\mathbf C^2_\ell\mathbf C^{\mathcal{F}}_{d,\sqcup}}\|1_A\|\ \geqslant\ \frac{1}{4\mathbf C_w\mathbf C^3_\ell\mathbf C^{\mathcal{F}}_{d,\sqcup}}\|1_{\varepsilon A}\|.
\end{align*}
Again since $B$ is a greedy set of $x+\sum_{n\in B}b_n e_n$, 
$$\left\|x+\sum_{n\in B}b_n e_n\right\|\ \geqslant\ \frac{1}{\mathbf C_\ell}\|x\|.$$
Therefore, we obtain 
$$2\left\|x+\sum_{n\in B}b_n e_n\right\|\ \geqslant\ \frac{1}{4\mathbf C_w\mathbf C^3_\ell\mathbf C^{\mathcal{F}}_{d}}\|1_{\varepsilon A}\|+\frac{1}{\mathbf C_\ell}\|x\|\ \geqslant\  \frac{1}{4\mathbf C_w\mathbf C^3_\ell\mathbf C^{\mathcal{F}}_{d}}\|1_{\varepsilon A} + x\|.$$
We have shown that 
$$\|x+ 1_{\varepsilon A}\|\ \leqslant\ 8\mathbf C_w\mathbf C^3_\ell\mathbf C^{\mathcal{F}}_{d}\left\|x+\sum_{n\in B}b_n e_n\right\|,$$
which completes our proof that $(e_n)$ has Property (A, $\mathcal{F}$).
\end{proof}

\begin{thm}\label{m2}
For a basis $(e_n)$ of a Banach space $X$, the following are equivalent:
\begin{enumerate}
    \item $(e_n)$ is $\mathcal{F}$-greedy,
    \item $(e_n)$ is $\mathcal{F}$-unconditional and has Property (A, $\mathcal{F}$),
    \item $(e_n)$ is $\mathcal{F}$-unconditional, $\mathcal{F}$-disjoint superdemocratic, and quasi-greedy,
    \item $(e_n)$ is $\mathcal{F}$-unconditional, $\mathcal{F}$-disjoint democratic, and quasi-greedy. 
\end{enumerate}
\end{thm}

\begin{proof}[Proof of Theorem \ref{m2}]
By Theorem \ref{m1}, we have that (1) $\Longleftrightarrow$ (2). Since an $\mathcal{F}$-greedy basis is quasi-greedy, and Property (A, $\mathcal{F}$) implies $\mathcal{F}$-disjoint superdemocracy (by definition), we get (1) $\Longleftrightarrow$ (2) $\Longrightarrow$ (3). Trivially, (3) $\Longrightarrow$ (4). That (4) $\Longrightarrow$ (2) is due to Proposition \ref{p20}. 
\end{proof}

\section{Characterizations of $\mathcal{F}$-almost greedy bases}

In this section, we first characterize $\mathcal{F}$-almost greedy bases using Property (A, $\mathcal{F}$), then show that the $\mathcal{F}$-almost greedy property is equivalent to the quasi-greedy property plus $\mathcal{F}$-disjoint superdemocracy.    

\begin{thm}\label{m6}
A basis $(e_n)$ is $C$-$\mathcal{F}$-almost greedy if and only if $(e_n)$ has $C$-Property (A, $\mathcal{F}$).
\end{thm}

\begin{proof}[Proof of Theorem \ref{m6}]
The proof that $C$-$\mathcal{F}$-almost greediness implies $C$-Property (A, $\mathcal{F}$) is similar to what we have in the proof of Theorem \ref{m1}. Conversely, assume that $(e_n)$ has $C$-Property (A, $\mathcal{F}$). Let $x\in \mathbb{X}$ with a greedy set $A$. 
Choose $B\in \mathcal{F}$ with $|B|\leqslant |A|$. 
If $A\backslash B = \emptyset$, then $A = B$ and $\|x-P_A(x)\| = \|x-P_B(x)\|$. If $A\backslash B\neq \emptyset$, 
note that $\min_{n\in A\backslash B}|e_n^*(x)|\geqslant \|x-P_A(x)\|_\infty$.
By Proposition \ref{p1}, we have
\begin{align*}
    \|x-P_A(x)\|&\ \leqslant\ C\|(x-P_A(x)) - P_{B\backslash A}(x) + P_{A\backslash B}(x)\|\\
    &\ =\ C\|x-P_B(x)\|.
\end{align*}
Since $B$ is arbitrary, we know that $(e_n)$ is $C$-$\mathcal{F}$-almost greedy. 
\end{proof}

\begin{thm}\label{m3}
Let $(e_n)$ be a basis. The following are equivalent:
\begin{enumerate}
    \item $(e_n)$ is $\mathcal{F}$-almost greedy,
    \item $(e_n)$ has Property (A, $\mathcal{F}$),
    \item $(e_n)$ is $\mathcal{F}$-disjoint superdemocratic and quasi-greedy,
    \item $(e_n)$ is $\mathcal{F}$-disjoint democratic and quasi-greedy. 
\end{enumerate}
\end{thm}

\begin{proof}[Proof of Theorem \ref{m3}]
That (1) $\Longleftrightarrow$ (2) follows from Theorem \ref{m6}. Clearly, an $\mathcal{F}$-almost greedy basis is quasi-greedy. By Proposition \ref{p20}, we have (2) $\Longleftrightarrow$ (4). Since (1) $\Longleftrightarrow$ (2) $\Longrightarrow$ (3) $\Longrightarrow$ (4), we are done. 
\end{proof}

\begin{cor}[Generalization of Theorem 2.3 in \cite{AA}]
A basis $(e_n)$ is $1$-$\mathcal{F}$-almost greedy if and only if $(e_n)$ has $1$-Property (A, $\mathcal{F}$).
\end{cor}

\section{Schreier families and $\mathcal{S}_\alpha$-greedy bases}\label{pc}

In this section, we will provide several non-trivial examples of $\mathcal{F}$-greedy basis. In particular, we will consider bases which are quasi-greedy but not greedy. As mentioned in the introduction, the Schreier families $\mathcal{S}_\alpha$ form a particularly rich collection of finite subsets of $\mathbb{N}$. 

\begin{proof}[Proof of Corollary \ref{m30'}]
Fix two countable ordinals $\alpha < \beta$. Let $N$ be as in Proposition \ref{k2}. Suppose that $(e_n)$ is $C$-$\mathcal{S}_\beta$-greedy for some constant $C\geqslant 1$. By Theorems \ref{m1'} and \ref{m1}, $(e_n)$ is  $C$-$\mathcal{S}_\beta$-suppression unconditional, $C$-$\mathcal{S}_\beta$-disjoint democratic, and $C$-suppression quasi-greedy. 

We show that $(e_n)$ is $C$-$\mathcal{S}_\alpha$-suppression unconditional. Let $x\in X$ and $E\in \mathcal{S}_\alpha$. We know that $E\backslash \{1, \ldots, N-1\}\in \mathcal{S}_\beta$. Hence, 
$$\|x-P_{E\backslash \{1, \ldots, N-1\}}(x)\|\ \leqslant\ C\|x\|.$$
We have
\begin{align*}
\|x-P_{E}(x)\|&\ \leqslant\ \|x-P_{E\backslash \{1, \ldots, N-1\}}(x)\|  + \|P_{E\cap \{1, \ldots, N-1\}}(x)\|\\
&\ \leqslant\ C\|x\| + N\sup_{n}\|e_n\|\|e_n^*\|\|x\|\ \leqslant\ (C+Nc_2^2)\|x\|.
\end{align*}
Therefore, $(e_n)$ is $\mathcal{S}_\alpha$-suppression unconditional. 

Next, we show that $(e_n)$ is $C$-$\mathcal{S}_\alpha$-disjoint democratic. Let $A\in \mathcal{S}_\alpha$ and $B\subset \mathbb{N}$ such that $A\cap B = \emptyset$ and $|A|\leqslant |B|$. Since $A\backslash \{1, \ldots, N-1\}\in\mathcal{S}_\beta$, we have
$$\|1_{A\backslash \{1, \ldots, N-1\}}\|\ \leqslant\ C\|1_B\|$$
Also, due to $C$-quasi-greediness, 
$$C\|1_B\|\ \geqslant\ c_1.$$
Hence,
\begin{align*}
    \|1_A\|&\ \leqslant\ \|1_{A\backslash \{1, \ldots, N-1\}}\| + \|1_{A\cap \{1, \ldots, N-1\}}\|\\
    &\ \leqslant\ C\|1_B\| + c_2N\ \leqslant\ C\|1_B\| + \frac{Cc_2N}{c_1}\|1_B\|\ =\ C\left(1+N\frac{c_2}{c_1}\right)\|1_B\|.
\end{align*}
Therefore, $(e_n)$ is $\mathcal{S}_\alpha$-disjoint democratic. 

By Theorem \ref{m1'}, we conclude that $(e_n)$ is $\mathcal{S}_\alpha$-greedy. 
\end{proof}

We have
$$\mbox{quasi-greedy}\ \Longleftarrow\ \mathcal{S}_\alpha\mbox{-greedy}\ \Longleftarrow\  \mathcal{S}_\beta\mbox{-greedy}\  \Longleftarrow \ \mbox{greedy}.$$
We construct bases to show that none of the reverse implications holds. Consider the following definition.

\begin{defi} \normalfont
Let $\omega_1$ denote the set of all countable ordinals and $(\alpha,\beta) \in (\omega_1\cup \{\infty\})^2$. A quasi-greedy basis $(e_n)$ for a Banach space $X$ is called $(\alpha,\beta)$-quasi-greedy if and only if $(e_n)$ is $\mathcal{S}_\alpha$-unconditional but not $\mathcal{S}_{\alpha+1}$-unconditional and $\mathcal{S}_\beta$-disjoint democratic but not $\mathcal{S}_{\beta+1}$-disjoint democratic. 

Suppose that either $\alpha$ or $\beta$ is $\infty$. If we denote by $\mathcal{S}_\infty$ the set of all finite subsets of $\mathbb{N}$, then $\mathcal{S}_\infty$-unconditionality and $\mathcal{S}_\infty$-disjoint democracy coincide with unconditionality and disjoint democracy, respectively.
\end{defi}

\begin{rek}\normalfont
Due to the proof of Corollary \ref{m30'}, a basis $(e_n)$ for a Banach space $X$ is $\mathcal{S}_{\eta}$-greedy if and only if it is $(\alpha,\beta)$-quasi-greedy for some $\alpha\geqslant \eta$ and $\beta \geqslant \eta$. Note also that the $(\infty,\infty)$-quasi-greedy property is the same as the greedy property, and a $(0,0)$-quasi-greedy basis is quasi-greedy but is far from being greedy.
\end{rek}

We prove Theorem \ref{m30} by providing the following examples.

\begin{thm}\label{zerozero}
There are spaces with bases $(e_n)$ that are $(0,0)$-quasi-greedy, $(\infty, 0)$-quasi-greedy, and $(0, \infty)$-quasi-greedy.  
\end{thm}

\begin{thm}\label{minfinity}
Fix a nonzero $\alpha\in \omega_1$. There is a space $X_{\alpha,\infty}$ with a basis $(e_n)$ that is $(\alpha,\infty)$-quasi-greedy. Hence, $X_{\alpha,\infty}$ is $\mathcal{S}_\alpha$-greedy but not $\mathcal{S}_{\alpha+1}$-greedy.
\end{thm}

\begin{thm}
Fix a nonzero $\alpha\in \omega_1$. There is a space $X_{\infty, \alpha}$ with a basis $(e_n)$ that is $(\infty, \alpha)$-quasi-greedy. Hence, $X_{\infty, \alpha}$ is $\mathcal{S}_\alpha$-greedy but not $\mathcal{S}_{\alpha+1}$-greedy. 
\end{thm}

\begin{rek}\normalfont
The bases we construct in Theorem \ref{minfinity} give new examples of conditional quasi-greedy bases. Furthermore, these bases are $1$-suppression quasi-greedy.  
\end{rek}

\subsection{Proof of Theorem \ref{zerozero}}

\subsubsection{A $(0,0)$-quasi-greedy basis}\label{1st}

We modify an example by Konyagin and Temlyakov \cite{KT1} who gave a conditional basis that is quasi-greedy. We shall construct a quasi-greedy basis that is neither $\mathcal{S}_1$-disjoint democratic nor $\mathcal{S}_1$-unconditional. 
For each $N\in\mathbb{N}$, let $X_N$ be the $(2N-1)$-dimensional space that is the completion of $c_{00}$ under the norm: for $x = (a_i)_i$,
$$\|(a_i)_i\|\ =\ \max\left\{\left(\sum_{i=1}^{2N-1}|a_i|^2\right)^{1/2}, \sup_{N\leqslant m\leqslant 2N-1}\left|\sum_{i=N}^m \frac{1}{\sqrt{i-N+1}}a_i\right|\right\}.$$
Let $X = (\oplus_{N=1}^\infty X_N)_{c_0}$. Let $\mathcal{B}$ be the canonical basis of $X$.

\begin{thm}
The basis $\mathcal{B}$ is $(0,0)$-quasi-greedy. 
\end{thm}

\begin{proof}
First, we show that $\mathcal{B}$ is not $\mathcal{S}_1$-unconditional. For each $X_N$, let $(f^N_i)_{i=1}^{2N-1}$ be the canonical basis of $X_N$ (that also belongs to $\mathcal{B}$). We have
$$\left\|\sum_{i=N}^{2N-1}\frac{1}{\sqrt{i-N+1}}f^N_i\right\|\ =\ \sum_{i=1}^N\frac{1}{i},\mbox{ while }\left\|\sum_{i=N}^{2N-1}\frac{(-1)^i}{\sqrt{i-N+1}}f^N_i\right\|\ =\ \left(\sum_{i=1}^N\frac{1}{i}\right)^{1/2}.$$
As $N\rightarrow\infty$, $\left\|\sum_{i=N}^{2N-1}\frac{1}{\sqrt{i-N+1}}f^N_i\right\|/\left\|\sum_{i=N}^{2N-1}\frac{(-1)^i}{\sqrt{i-N+1}}f^N_i\right\|\rightarrow\infty$; hence, $\mathcal{B}$ is not $\mathcal{S}_1$-unconditional. 

Next, we show that $\mathcal{B}$ is not $\mathcal{S}_1$-disjoint democratic. We have
$$\left\|\sum_{i=N}^{2N-1}f^N_i\right\|\ =\ \sum_{i=1}^{N}\frac{1}{\sqrt{i}}, \mbox{ while }\left\|\sum_{i=N+1}^{2N} f^i_1\right\|\ =\ 1.$$
Therefore, $\mathcal{B}$ is not $\mathcal{S}_1$-disjoint democratic. 

Finally, we prove that $\mathcal{B}$ is quasi-greedy. To do so, we need only to show that for each $N$, the basis $(f^N_i)_{i=1}^{2N-1}$ has the same quasi-greedy constant of $3 + \sqrt{2}$. Let $(a_i)_{i=1}^{2N-1}\in X_N$, where $\|(a_i)_i\|\leqslant 1$.
It suffices to prove that 
$$\left|\sum_{i\in \Lambda}\frac{1}{\sqrt{i-N+1}}a_i\right|\ \leqslant\ 3+\sqrt{2},$$
for all $\varepsilon > 0$, for all $M\in [N, 2N-1]$, and $\Lambda = \{N\leqslant i\leqslant M: |a_i|>\varepsilon\}$. Since $\|(a_i)_i\|\leqslant 1$, we know that $|a_i|\leqslant 1$ and so, we can assume that $0 < \varepsilon < 1$. Set $L = \lfloor \varepsilon^{-2}\rfloor$ to have $1/2 \leqslant \varepsilon^2 L\leqslant 1$. We proceed by case analysis.

Case 1: $M-N+1\leqslant L$. We have
\begin{align*}
    \left|\sum_{i\in \Lambda}\frac{a_i}{\sqrt{i-N+1}}\right|&\ \leqslant\ \left|\sum_{N\leqslant i\leqslant M}\frac{a_i}{\sqrt{i-N+1}}\right| + \left|\sum_{\substack{N\leqslant i\leqslant M\\|a_i| \leqslant \varepsilon}}\frac{a_i}{\sqrt{i-N+1}}\right|\\
    &\ \leqslant\ 1 + \varepsilon \sum_{i=N}^{M}\frac{1}{\sqrt{i-N+1}}\\
    &\ \leqslant\ 1 + \varepsilon \sum_{i=1}^{M-N+1}\frac{1}{\sqrt{i}}\\
    &\ \leqslant\ 1 + 2\varepsilon\sqrt{M-N+1}\ \leqslant\ 1 + 2\varepsilon \sqrt{L}\ \leqslant\ 3.
\end{align*}

Case 2: $M-N+1 > L$. We have
$$
\left|\sum_{i\in \Lambda}\frac{a_i}{\sqrt{i-N+1}}\right|\ =\ \left|\sum_{\substack{N\leqslant i\leqslant N+L-1\\ |a_i|> \varepsilon}}\frac{a_i}{\sqrt{i-N+1}}\right| + \left|\sum_{\substack{N+L\leqslant i\leqslant M\\|a_i| > \varepsilon}}\frac{a_i}{\sqrt{i-N+1}}\right|.
$$
By above, 
$$ \left|\sum_{\substack{N\leqslant i\leqslant N+L-1\\ |a_i|> \varepsilon}}\frac{a_i}{\sqrt{i-N+1}}\right|\ \leqslant\ 3.$$
Furthermore, we have
\begin{align*}
    \left|\sum_{\substack{N+L\leqslant i\leqslant M\\|a_i| > \varepsilon}}\frac{a_i}{\sqrt{i-N+1}}\right|&\ \leqslant\ \left(\sum_{N+L\leqslant i\leqslant M}\frac{1}{(i-N+1)^{3/2}}\right)^{1/3}\left(\sum_{\substack{N+L\leqslant i\leqslant M\\|a_i|>\varepsilon}}|a_i|^{3/2}\right)^{2/3}\\
    &\ \leqslant\ \left(\sum_{i=L+1}^\infty \frac{1}{i^{3/2}}\right)^{1/3}\left(\sum_{\substack{N+L\leqslant i\leqslant M\\|a_i| > \varepsilon}}|a_i|^{3/2}\sqrt{\frac{|a_i|}{\varepsilon}}\right)^{2/3}\\
    &\ \leqslant\ 2^{1/3}L^{-1/6}\varepsilon^{-1/3}\ \leqslant \ \sqrt{2}.
\end{align*}
This completes our proof. 
\end{proof}

\subsubsection{An $(\infty, 0)$-quasi-greedy basis}
Define $$\mathcal{F} \ :=\ \{A\subset\mathbb{N}: A\mbox{ is finite and does not contain even integers}\}.$$
Let $\mathbb{X}$ be the completion of $c_{00}$ with respect to the following norm: for $x = (x_1, x_2, \ldots)$,  let 
$$\|x\|\ :=\ \left(\sum_{2|i}|x_i|\right) + \left(\sum_{2\nmid i}|x_i|^2\right)^{1/2}.$$
Let $\mathcal{B}$ be the canonical basis. Clearly, $\mathcal{B}$ is $1$-unconditional. Note that $\mathcal{B}$ is not $\mathcal{S}_1$-disjoint democratic. To see this, fix $N\in\mathbb{N}$ and choose $A = \{1, 3, 5, \ldots, 2N-1\}$ and $B = \{2N, 2N+2, 2N+4, \ldots, 4N-2\}\in\mathcal{S}_1$. Then $\|1_A\| = \sqrt{N}$ while $\|1_B\| = N$. Hence, $\|1_B\|/\|1_A\|\rightarrow \infty$ as $N\rightarrow \infty$. It follows that $\mathcal{B}$ is not $\mathcal{S}_1$-disjoint democratic.

\subsubsection{A $(0,\infty)$-quasi-greedy basis}
We define the spaces $X_N$ as in Subsection \ref{1st}: for each $N\in\mathbb{N}$, let $X_N$ be the $(2N-1)$-dimensional space that is the completion of $c_{00}$ under the norm: for $x = (a_i)_i$,
$$\|(a_i)_i\|\ =\ \max\left\{\left(\sum_{i=1}^{2N-1}|a_i|^2\right)^{1/2}, \sup_{N\leqslant m\leqslant 2N-1}\left|\sum_{i=N}^m \frac{1}{\sqrt{i-N+1}}a_i\right|\right\}.$$
Let $X = (\oplus_{N=1}^\infty X_N)_{\ell_2}$. Let $\mathcal{B}$ be the canonical basis of $X$. Using the same argument as in Subsection \ref{1st}, we know that $\mathcal{B}$ is quasi-greedy and is not $\mathcal{S}_1$-unconditional. We show that $\mathcal{B}$ is democratic. Let $A\subset\mathcal{B}$ be a nonempty finite set. Write $A = \bigcup_{N=1}^\infty A_N$, where $A_N$ is the intersection of $A$ and the canonical basis of $X_N$. We have
$$\left\|\sum_{e\in A} e\right\|\ =\ \left(\sum_{N=1}^\infty\left\|\sum_{e\in A_N} e\right\|^2\right)^{1/2}\ \geqslant\ \left(\sum_{N=1}^\infty|A_N|\right)^{1/2}\ =\ |A|^{1/2}.$$
On the other hand, for each $N$, 
$$\left\|\sum_{e\in A_N}e\right\|\ \leqslant\ \sum_{i=1}^{|A_N|}\frac{1}{\sqrt{i}}\ \leqslant\ 2\sqrt{|A_N|}.$$
Therefore,
$$\left\|\sum_{e\in A} e\right\|\ =\ \left(\sum_{N=1}^\infty\left\|\sum_{e\in A_N} e\right\|^2\right)^{1/2}\ \leqslant\ 2\left(\sum_{N=1}^\infty |A_N|\right)^{1/2}\ =\ 2|A|^{1/2}.$$
We have shown that $|A|^{1/2}\leqslant \|\sum_{e\in A}e\|\leqslant 2|A|^{1/2}$, so $\mathcal{B}$ is democratic. 

\subsection{An $(\alpha,\infty)$-quasi-greedy basis} \label{rt}

Fix a nonzero $\alpha\in \omega_1$ and consider the following collection subsets related to $\mathcal{S}_\alpha$
\begin{equation*}
\mathcal{F}_{\alpha} = \{\cup^{r}_{i=1} E_i:r/2\leqslant E_1<E_2<\cdots <E_{r} \mbox{ are in } \mathcal{S}_{\alpha-1}\}.
\end{equation*}
The family $\mathcal{F}_1$ (among others) recently appeared in \cite{BCF}. 

\begin{lem}\label{ls}
Let $F \in \mathcal{F}_{\alpha}$. Then $F$ can be written as the union of two disjoint sets in $\mathcal{S}_\alpha$.
\end{lem}

\begin{proof}
Write $F = \cup_{i=1}^r E_i$, where $r/2 \leqslant E_1 < E_2 < \cdots < E_{r}$ and sets $E_i\in \mathcal{S}_{\alpha-1}$. 
Discard all the empty $E_i$ and re-number to have nonempty sets $E'_i$ satisfying $r/2 \leqslant E_1' < E'_2 < \cdots < E'_\ell$ for some $\ell\leqslant r$.  Let $s = \lceil r/2\rceil$.

Case 1: $s\geqslant \ell$. Then $s\leqslant E_1' < E_2' < \cdots < E'_\ell$ implies that $F = \cup_{i = 1}^\ell E_i'\in \mathcal{S}_\alpha$. We are done.

Case 2: $s < \ell$. Let $F_1 = \cup_{i=1}^s E'_i$, which is in $S_\alpha$ due to Case 1. Note that 
$$s+1 \ \leqslant\ E'_{s+1}  \ <\ \cdots \ <\ E'_\ell;$$
furthermore, $\ell-s \leqslant r-s \leqslant s+1$. Therefore, $F_2 := \cup_{i=s+1}^\ell E'_i\in \mathcal{S}_\alpha$. 
Since $F = F_1\cup F_2$, we are done. 
\end{proof}

Clearly, $\mathcal{S}_{\alpha}\subset\mathcal{F}_{\alpha}$. Let $X_{\alpha,\infty}$ be the completion of $c_{00}$ under the following norm: for $(a_i)\in c_{00}$,
$$\|(a_i)\|_{X_{\alpha,\infty}} :=\ \sup\left\{\sum_{j=1}^d \left|\sum_{i\in I_j}a_i\right|: I_1 < I_2 < \cdots < I_d \mbox{ intervals}, (\min I_j)_{j=1}^d\in \mathcal{F}_{\alpha}\right\}.$$
The space $X_{\alpha,\infty}$ above is the Jamesfication of the combinatorial space $X[\mathcal{F}_\alpha]$ (see \cite{AMS, BHO}) and is denoted by $J(X[\mathcal{F}_\alpha])$. 

\begin{thm}
The standard basis $(e_n)$ for the space $X_{\alpha,\infty}$ is $(\alpha,\infty)$-quasi-greedy.
\end{thm}

We prove the above theorem through the following propositions. Let us start with the easiest one.

\begin{prop}\label{pp2}
 The basis $(e_n)$ is democratic and  $\mathcal{F}_\alpha$-unconditional, and thus $\mathcal{S}_\alpha$-unconditional.

\end{prop}

\begin{proof}
It follows directly from the definition of $\|\cdot\|$ that for $x\in X$ and $F \in \mathcal{F}_\alpha$,
$$\left\|\sum_{i\in F} e_i^*(x)e_i\right\|_{X_{\alpha, \infty}}\ =\ \sum_{i \in F}|e_i^*(x)|\ \leqslant\ \|x\|_{X_{\alpha, \infty}}.$$ 
Hence, $(e_n)$ is $\mathcal{F}_\alpha$-unconditional.

 Let $A, B\subset\mathbb{N}$ with $|A| \leqslant |B|$. By Proposition \ref{k2}, there exists $N\in\mathbb{N}_{\geqslant 6}$ such that $$E\backslash \{1, \ldots, N-1\}\in \mathcal{F}_\alpha, \forall E\in \mathcal{S}_1.$$ 
 Without loss of generality, assume that $|B|\geqslant N^2$. Let $B'\subset B$ such that $|B'|\geqslant |B|/2$ and $B'\in \mathcal{S}_1\subset \mathcal{F}_1$. Form $B'' = B'\backslash \{1, \ldots, N-1\}\in \mathcal{F}_\alpha$. We have
$$\|1_B\|\ \geqslant\  |B''|\ \geqslant\ |B'| - N\ \geqslant\ |B|/3 \ \geqslant\ |A|/3\ \geqslant\ \|1_A\|/3.$$
Therefore, $(e_n)$ is democratic. 
\end{proof}

\begin{prop}\label{pp1}
The basis $(e_n)$ for the space $X_{\alpha,\infty}$ is 1-suppression quasi-greedy.
\end{prop}

\begin{proof}
Let $x=(a_i)\in X_{\alpha,\infty}$ and $|a_N| = \|x\|_\infty$. By induction, we need only to show that 
$$\|x-a_Ne_N\|\ \leqslant\ \|x\|.$$ Suppose, for a contradiction, that $\|x-a_Ne_N\| > \|x\|$. Removing the $N$th coefficient increases the norm implies that there exists an admissible set of intervals $\{I_j\}_{j=1}^d$ satisfying
\begin{enumerate}
    \item $a_{\min I_j}a_{\max I_j}\neq 0$ for all $1\leqslant j\leqslant d$,
    \item for some $k$, $N\in I_k$ and $\min I_k < N < \max I_k$,
    \item $\sum_{1\leqslant j\leqslant d, j\neq k} |\sum_{i\in I_j}a_i| + |\sum_{i\in I_k, i\neq N}a_i| > \|x\|$.
\end{enumerate} For two integers $a \leqslant b$, let $[a, b] = \{a, a+1, \ldots, b\}$; when $a > b$, we let $[a, b] = \emptyset$.
We form a new sequence of intervals as follows: if $k > 1$, 
\begin{align*}&I_1' \ =\ I_1\backslash \min I_1, I'_2 \ =\ I_2, \ldots, I'_{k-1} \ =\ I_{k-1},\\
&I'_k \ =\ [\min I_k, N-1], I'_{k+1} \ =\ \{N\}, I'_{k+2} \ =\ [N+1, \max I_k],\\
&I'_{k+3}\ =\ I_{k+1}, \ldots, I'_{d+2}\ =\ I_{d}.
\end{align*}
If $k= 1$, then
\begin{align*}
&I'_1 \ =\ [\min I_1 + 1, N-1], I'_2 \ =\ \{N\}, I'_{3}\ =\ [N+1, \max I_1],\\
&I'_{4}\ =\ I_{2}, \ldots, I'_{d+2}\ =\ I_d.
\end{align*}

To see that $\{I_j'\}_{j=1}^{d+2}$ is admissible, we need to show $\{\min I_j'\}_{j=1}^{d+2}\in \mathcal{F}_\alpha$. We consider only the case when $k > 1$; the case $k = 1$ is similar.
By construction, 
$$\{\min I_j'\,:\,1\leqslant j\leqslant d+2\} \ =\ \{\min (I_1\backslash \min I_1)\}\cup \{\min I_j\,:\, 2\leqslant j\leqslant d\}\cup\{N, N+1\}.$$
Let $A = \{\min I_j\}_{j=1}^d$ and $B = \{\min (I_1\backslash \min I_1)\}\cup \{\min I_j\,:\, 2\leqslant j\leqslant d\}$. Since $\min B-\min A\geqslant 1$ and $A\in \mathcal{F}_\alpha$, we know that $B\cup \{N, N+1\}\in \mathcal{F}_\alpha$.

We now use the admissible set $(I'_j)_{j=1}^{d+2}$ to obtain a contradiction. Write
\begin{equation}\label{ee1}\|x\|\ \geqslant\ \sum_{j=1}^{d+2}\left|\sum_{i\in I'_j}a_i\right| \ =\ \sum_{j = 1, k , k+1, k+2}\left|\sum_{i\in I'_j}a_i\right| + \sum_{j\neq 1, k, k+1, k+2}\left|\sum_{i\in I'_j}a_i\right|.\end{equation}
Since $|a_N|\geqslant |a_{\min I_1}|$, we have
\begin{align}\label{ee2}\sum_{j = 1, k , k+1, k+2}\left|\sum_{i\in I'_j}a_i\right|&\ \geqslant\ \left(\left|\sum_{i\in I_1}a_i\right|-|a_{\min I_1}|\right) + \left|\sum_{i=\min I_k}^{N-1}a_i\right| + |a_N| + \left|\sum_{i=N+1}^{\max I_k}a_i\right|\nonumber\\
&\ \geqslant\ \left|\sum_{i\in I_1}a_i\right| + \left|\sum_{i\in I_k, i\neq N}a_i\right|.
\end{align}
Furthermore, by definition,  
\begin{equation}\label{ee3}\sum_{j\neq 1, k, k+1, k+2}\left|\sum_{i\in I'_j}a_i\right|\ =\ \sum_{j = 2}^{k-1}\left|\sum_{i\in I_j}a_i\right| + \sum_{j=k+1}^{d}\left|\sum_{i\in I_j}a_i\right|.\end{equation}
By \eqref{ee1}, \eqref{ee2}, and \eqref{ee3}, we conclude that 
$$\|x\|\ \geqslant\ \sum_{1\leqslant j\leqslant d, j\neq k} \left|\sum_{i\in I_j}a_i\right| + \left|\sum_{i\in I_k, i\neq N}a_i\right|\ > \ \|x\|,$$
which is a contradiction. Therefore, $(e_n)$ is a $1$-suppression quasi-greedy. 
\end{proof}

\begin{cor}
The basis $(e_n)$ is $\mathcal{F}_\alpha$-greedy and thus, is $\mathcal{S}_\alpha$-greedy.
\end{cor}
\begin{proof}
Use Theorem \ref{m2} and Propositions \ref{pp2} and \ref{pp1}. 
\end{proof}

It remains to show that $(e_n)$ is not $\mathcal{S}_{\alpha+1}$-unconditional and thus, not $\mathcal{S}_{\alpha+1}$-greedy. 
This part of the proof will require the repeated averages hierarchy \cite[pp. 1053]{AGR}. However, for our purposes, we only need the following lemma, a weaker result than \cite[Proposition 12.9]{AT}. 

\begin{lem}
For each $\alpha\in \omega_1$, $\varepsilon>0$ and $N\in \mathbb{N}$, there is a sequence $(a^\alpha_k)_{k=1}^\infty$ satisfying
\begin{enumerate}
    \item $a_k^\alpha \geqslant 0$ for each $k\in \mathbb{N}$ and $\|(a^\alpha_k)_k\|_{\ell_1}=1$,
    \item $\{k : a^\alpha_k\neq 0\}$ is an interval and a maximal  $\mathcal{S}_{\alpha+1}$-set,
    \item $L:= \min \{k: a^\alpha_k\neq 0\} > N$ and $(a^\alpha_k)_{k\geqslant L}$ is monotone decreasing,
    \item for each $G \in \mathcal{S}_\alpha$, we have
    $\sum_{k \in G} a_k^{\alpha} <\varepsilon$. 
\end{enumerate}
\label{RAH}
\end{lem}

Choose $N$ such that $$E\backslash \{1, \ldots, N-1\}\in \mathcal{S}_\alpha, \forall E\in \mathcal{S}_1.$$
Fix $\varepsilon>0$ and find $(a_k^\alpha)$ satisfying Lemma \ref{RAH} with $N$ chosen as above. Since $F=\{k : a_k^\alpha\not=0\}\in \mathcal{S}_{\alpha+1}$, write $F = \cup_{i=1}^m E_i$, where $m\leqslant E_1 < E_2 <\cdots < E_m$ and $E_i\in \mathcal{S}_\alpha$. Since $F$ is an interval, each $E_i$ is an interval; furthermore, $N < \{\min E_i: 1\leqslant i\leqslant m\}\in \mathcal{S}_1$. Hence, $\{\min E_i: 1\leqslant i\leqslant m\}\in \mathcal{S}_\alpha\subset \mathcal{F}_\alpha$. By Lemma \ref{RAH} items (1) and (2), we have $\|\sum_{k\in F}a_k^\alpha e_k\| = 1$. 

We estimate $\sum_{k\in F}(-1)^k a_k^\alpha e_k$. Let $I_1<\cdots < I_d$ be intervals so that $(\min I_j)_{j=1}^d \in \mathcal{F}_\alpha$ and $a^\alpha_{\min I_j}\neq 0$. For any interval $I_j$, $|\sum_{i\in I_j}(-1)^k a_k^\alpha|\leqslant 2a^\alpha_{\min I_j}$ because $(a^\alpha_k)_k$ is monotone decreasing. Therefore,
$$\sum_{j=1}^d \left|\sum_{k\in I_j}(-1)^k a_k^\alpha\right|\ \leqslant\ \sum_{j=1}^d 2 a^\alpha_{\min I_j}.$$
By Lemma \ref{ls}, we can write the set $\{\min I_1, \min I_{2}, \ldots, \min {I_d}\}$ as the union of two disjoint sets $A_1$ and $A_2$ in $\mathcal{S}_\alpha$. By Lemma \ref{RAH} item (3), we obtain
$$\sum_{j=1}^d a^\alpha_{\min I_j}\ =\ \sum_{i\in A_1}a^\alpha_i + \sum_{i\in A_2}a^\alpha_i\ <\ 2\varepsilon.$$
Thus $\|\sum_{k\in F}(-1)^k a_k^\alpha e_k\| < 4\varepsilon$. As $\varepsilon$ was arbitrary and $F\in \mathcal{S}_{\alpha+1}$, we see that $(e_n)$ is not $\mathcal{S}_{\alpha+1}$-unconditional.

\subsection{An $(\infty,\alpha)$-quasi-greedy basis}

\subsubsection{Repeated average hierarchy}

Let $[\mathbb{N}]$ denote the collection of all infinite subsequences of $\mathbb{N}$. Similarly, if $M\in [\mathbb{N}]$, then $[M]$ denotes the collection of all infinite subsequences of $M$. 

\begin{defi}\normalfont Let $\mathcal{B} = (e_n)$ be the canonical basis of $c_{00}$. 
For every countable ordinal $\alpha$ and $M = (m_n)_{n=1}^\infty\in [\mathbb{N}]$, we define a convex block sequence $(\alpha(M, n))_{n=1}^\infty$ of $\mathcal{B}$ by transfinite induction on $\alpha$. If $\alpha = 0$, then $\alpha(M, n) := e_{m_n}$. Assume that $(\beta(M, n))_{n=1}^\infty$ has been defined for all $\beta < \alpha$ and all $M\in [\mathbb{N}]$. For $M\in [\mathbb{N}]$, we define $(\alpha(M, n))_{n=1}^\infty$.

If $\alpha$ is a successor ordinal, write $\alpha = \beta + 1$. Set 
$$\alpha(M, 1)\ :=\ \frac{1}{m_1}\sum_{n=1}^{m_1}\beta(M, n).$$
Suppose that $\alpha(M,1) < \cdots < \alpha(M,n)$ have been defined. Let 
$$M_{n+1} \ :=\ \{m\in M: m > \max\supp (\alpha(M,n))\}\mbox{ and }k_n \ :=\ \min M_{n+1}.$$
Set 
$$\alpha(M, n+1)\ :=\ \frac{1}{k_n}\sum_{i=1}^{k_n} \beta(M_{n+1}, i).$$

If $\alpha$ is a limit ordinal, let $(\alpha_n+1)\nearrow \alpha$. Set 
$$\alpha(M,1)\ :=\ (\alpha_{m_1}+1)(M, 1).$$
Suppose that $\alpha(M,1) < \cdots < \alpha(M,n)$ have been defined. Let 
$$M_{n+1} \ :=\ \{m\in M: m > \max\supp (\alpha(M,n))\}\mbox{ and }k_n \ :=\ \min M_{n+1}.$$
Set
$$\alpha(M, n+1) \ :=\ (\alpha_{k_n}+1)(M_{n+1}, 1).$$
\end{defi}

\begin{lem}\label{lh1}
For each ordinal $\alpha\geqslant 1$ and $M\in [\mathbb{N}]$, we have
\begin{equation}\label{eh1}\|\alpha(M, n)\|_{\ell_1} \ =\ 1\mbox{ and }0\ \leqslant\ e_i^*(\alpha(M, n)) \ \leqslant\ \frac{1}{\min\supp(\alpha(M, n))}, \forall n, i\in \mathbb{N}.\end{equation}
\end{lem}

\begin{proof}
The proof is immediate from induction. 
\end{proof}

\begin{prop}\label{ppp1}
Fix $\alpha < \beta$. For all $N\in \mathbb{N}$ and $M\in [\mathbb{N}]$, there exists $L\in [M]$ such that $\min L > N$ and 
$$\|\beta(L, 1)\|_{\alpha} \ <\ \frac{3}{\min L},$$
where 
$$\|(a_n)\|_{\alpha}\ :=\ \sup_{F\in \mathcal{S}_\alpha}\sum_{n\in F}|a_n|.$$
\end{prop}

\begin{rek}\normalfont
See \cite[Proposition 2.3]{AT} for the case when $\alpha$ is a finite ordinal. Our proof of Proposition \ref{ppp1} is a combination of ideas used in the proofs of \cite[Proposition 2.3]{AT} and \cite[Proposition 2.15]{AG}.
\end{rek}

\begin{proof}[Proof of Proposition \ref{ppp1}]
We prove by transfinite induction on $\beta$. Base case: $\beta = 1$. Then $\alpha = 0$. Let $N\in \mathbb{N}$ and $M = (m_n)_{n=1}^\infty\in [\mathbb{N}]$. Let $m_k$ be the smallest such that $m_k > N$. Choose $L = (m_n)_{n\geqslant k}$. We have
$$\|1(L, 1)\|_{0}\ =\ \frac{1}{\min L}\ <\ \frac{3}{\min L}.$$
Indeed, for finite ordinals $\beta\geqslant 1$, we know the conclusion holds by \cite[Proposition 2.3]{AT}.
Inductive hypothesis: suppose that the statement holds for all $\eta < \beta$ for some $\beta\geqslant \omega$. We need to show that it also holds for $\beta$.

Case 1: $\beta$ is a limit ordinal. Let $(\beta_n + 1)\nearrow \beta$ and $\alpha < \beta$. Choose $m > N$ such that $\beta_m > \alpha$. Set $L_1 :=  M|_{> m}$ and $\ell := \min L_1 > m$. Note that $\ell\geqslant 3$.
By the inductive hypothesis, there exists $L_2\in [M]$ such that $\min L_2 > \max\supp(\beta_\ell(L_1, 1))$ and 
$$\|\beta_\ell(L_2, 1)\|_\alpha \ <\ \frac{3}{\min {L_2}}.$$
Repeat the process to find subsequences $L_3, \ldots, L_\ell\in [M]$ such that 
$$\supp(\beta_\ell(L_1, 1)) \ <\ \supp(\beta_\ell(L_2, 1)) \ <\ \cdots < \ \supp(\beta_\ell(L_\ell, 1))$$ 
and 
$$\|\beta_\ell(L_n, 1)\|_\alpha \ <\ \frac{3}{\min L_n}, \forall~ 2\leqslant n\leqslant \ell.$$
Let $L:= \cup_{n=1}^{\ell-1}\supp (\beta_\ell(L_n, 1))\cup L_\ell\in [M]$. Then $\min L > N$. By definition,
$$\beta(L, 1)\ :=\ (\beta_{\ell}+1)(L, 1)\ =\  \frac{1}{\ell}\sum_{n=1}^\ell \beta_\ell(L, n)\ =\ \frac{1}{\ell} \sum_{n=1}^\ell \beta_\ell(L_n, 1).$$
We have 
\begin{align*}
\|\beta(L, 1)\|_\alpha&\ \leqslant\ \frac{1}{\ell} \sum_{n=1}^\ell \|\beta_\ell(L_n, 1)\|_\alpha\\
&\ \leqslant\  \frac{1}{\ell}+ \frac{1}{\ell} \left(\frac{3}{\min L_2} + \cdots + \frac{3}{\min L_\ell}\right)\\
&\ \leqslant\ \frac{1}{\ell} + \frac{1}{\ell} \frac{3}{\min L_2}\left(1 + \frac{1}{8} + \frac{1}{8^2} + \cdots\right)\mbox{ by Lemma \ref{8times}}\\
&\ =\ \frac{1}{\ell}\left(1+\frac{24}{7\min L_2}\right) \ <\ \frac{3}{\ell}.
\end{align*}

Case 2: $\beta$ is a successor ordinal. Write $\beta = \eta + 1$. 

\begin{enumerate}
    \item Case 2.1: $\alpha < \eta$. Set $L_1:= M|_{>{N+1}}$ and $\ell := \min L_1 \geqslant 3$. 
By the inductive hypothesis, there exists $L_2\in [M]$ such that $\min L_2 > \max\supp(\eta(L_1, 1))$ and 
$$\|\eta(L_2, 1)\|_\alpha \ <\ \frac{3}{\min L_2}.$$
Repeat the process to find subsequences $L_3, \ldots, L_\ell$ such that 
$$\supp(\eta(L_1, 1)) \ <\ \supp(\eta(L_2, 1)) \ <\ \cdots < \ \supp(\eta(L_\ell, 1))$$ 
and 
$$\|\eta(L_n, 1)\|_\alpha \ <\ \frac{3}{\min L_n}, \forall 2\leqslant n\leqslant \ell.$$
Let $L:= \cup_{n=1}^{\ell-1}\supp (\eta(L_n, 1))\cup L_\ell\in [M]$. Then $\min L > N$. By definition,
$$\beta(L, 1)\ :=\ (\eta+1)(L, 1)\ =\  \frac{1}{\ell}\sum_{n=1}^\ell \eta(L, n)\ =\ \frac{1}{\ell} \sum_{n=1}^\ell \eta(L_n, 1).$$
Similar to Case 1, we have $\|\beta(L, 1)\|_\alpha < 3/\ell$.
\item Case 2.2: $\alpha = \eta$. Let $(\alpha_n+1)\nearrow \alpha$ and $\mathcal{S}_{\alpha_n}\subset \mathcal{S}_{\alpha_{n+1}}$ for all $n\geqslant 1$. Set $L_1:= M|_{>{N+1}}$ and $\ell := \min L_1 \geqslant 3$. We have  
$$(\alpha_\ell+1)(L_1, 1)\ =\ \alpha(L_1, 1).$$
Let $k_1 = \max\supp (\alpha(L_1, 1))$. By the inductive hypothesis, find $L_2\in [M]$ with $k_1 < \min L_2$ and
$$\|\alpha(L_2, 1)\|_{\alpha_{k_1}} \ <\ \frac{3}{\min L_2}.$$
Repeat the process to find subsequences $L_3, \ldots, L_\ell\in [M]$ such that 
$$\supp(\alpha(L_1, 1)) \ <\ \supp(\alpha(L_2, 1)) \ <\ \cdots < \ \supp(\alpha(L_\ell, 1))$$ 
and if $k_n = \max\supp (\alpha(L_n, 1))$, we have 
$$\|\alpha(L_n, 1)\|_{\alpha_{k_{n-1}}} \ <\ \frac{3}{\min L_n}, \forall 2\leqslant n\leqslant \ell.$$
Let $L:= \cup_{n=1}^{\ell-1}\supp (\alpha(L_n, 1))\cup L_\ell\in [M]$. Then
$\beta(L, 1)\ :=\ \frac{1}{\ell}\sum_{n=1}^\ell \alpha(L_n, 1)$. 

It holds that
$\|\beta(L, 1)\|_\alpha < \frac{3}{\ell}$.
Indeed, let $G\in \mathcal{S}_\alpha$. Suppose that $k:=\min G\in \supp(\alpha(L_{j_0}, 1))$. Then $k\leqslant k_{j_0}$. By the definition of $\mathcal{S}_\alpha$, choose $p\leqslant k$ such that $G\in \mathcal{S}_{\alpha_p + 1}$. Finally, let $q\leqslant k$ be such that $G = \cup_{n=1}^q G_n$, where $G_1 < G_2 < \cdots < G_q$ and $G_n\in \mathcal{S}_{\alpha_p}$. For $j_0< n \leqslant \ell$, because $p\leqslant k\leqslant k_{n-1}$, we obtain $\mathcal{S}_{\alpha_p}\subset \mathcal{S}_{\alpha_{k_{n-1}}}$ and
$$\|\alpha(L_n, 1)\|_{\alpha_p}\ \leqslant\ \|\alpha(L_n, 1)\|_{\alpha_{k_{n-1}}} \ <\ \frac{3}{\min L_n}.$$
Therefore, 
$$\sum_{n\in G}e_n^*(\alpha(L_n, 1))\ \leqslant\ q\frac{3}{\min L_n}, \forall j_0< n \leqslant \ell.$$
Noting that $q \leqslant k\leqslant k_{j_0}< \min L_{j_0+1}\leqslant \frac{1}{8} \min L_{j_0+2}$ by Lemma \ref{8times}, we have
\begin{align*}\sum_{n\in G}e_n^*(\beta(L, 1))&\ =\ \frac{1}{\ell}\left(1+1+ 3q\sum_{n=j_0+2}^\ell \frac{1}{\min L_n}\right)\\
&\ \leqslant\ \frac{1}{\ell}\left(2 + \frac{24q}{7\min L_{j_0+2}}\right)\ <\ \frac{3}{\ell}. 
\end{align*}
\end{enumerate}
We have completed the proof. 
\end{proof}

\subsubsection{An $(\infty, \alpha)$-quasi-greedy basis}

By Proposition \ref{ppp1}, we can find infinitely many $\mathcal{S}_{\alpha+1}$-maximal sets 
$F_1 < F_2 < F_3 < \cdots $ and for each set $F_i$, coefficients $(w_n)_{n\in F_i}$,  such that
$\sum_{n\in F_i} w_n = 1$, while $$\left\|\min F_i\cdot \sum_{n\in F_i}w_ne_n\right\|_{\alpha}\ <\ 3.$$

Let $X$ be the completion of $c_{00}$ under the norm: 
$$\|(a_n)_n\|\ :=\ \sup_{F_i}\left\{\max_{n}|a_n|, \min F_i\cdot \sum_{n\in F_i} w_n|a_n|\right\}.$$
Let $\mathcal{B}$ be the canonical basis.

\begin{claim}\label{cl2}
The basis $\mathcal{B}$ is $1$-unconditional and normalized. 
\end{claim}
\begin{proof}
That $\mathcal{B}$ is $1$-unconditional is obvious. Let us show that $\|e_n\| = 1$ for all $n\in \mathbb{N}$. Fix $n\in \mathbb{N}$. Due to the appearance of $\|\cdot\|_\infty$, $\|e_n\|\geqslant 1$. Since $\min F_i \cdot w_n \leqslant 1$ for all $i\in \mathbb{N}$ and $n\in F_i$ according to Lemma \ref{lh1}, $\|e_n\|\leqslant 1$. Hence, $\|e_n\| = 1$.
\end{proof}

\begin{claim}\label{cl3}
The basis $\mathcal{B}$ is $\mathcal{S}_\alpha$-disjoint democratic. In particular, $\|1_A\| < 3$ for all $A\in \mathcal{S}_\alpha$. 
\end{claim}

\begin{proof}
Choose $A\in \mathcal{S}_\alpha$. For any $F_i$, we have
\begin{align*}
    \min F_i\cdot \sum_{n\in A\cap F_i}w_n \ \leqslant\ \left\|\min F_i\cdot \sum_{n\in F_i}w_ne_n\right\|_{\alpha}\ < \ 3.
\end{align*}
Therefore, $\|1_A\| < 3$.
\end{proof}

\begin{claim}\label{cl4}
The basis $\mathcal{B}$ is not $\mathcal{S}_{\alpha+1}$-disjoint democratic. 
\end{claim}

\begin{proof}
Choose $F_i$, which is a maximal $\mathcal{S}_{\alpha+1}$-set.
Let $A$ be an $\mathcal{S}_\alpha$-set with $|F_i| \leqslant |A|$ and $F_i\sqcup A$. By how $F_i$'s are defined, $\|1_{F_i}\| = \min F_i$.
On the other hand, we have that $\|1_A\| < 3$ by Claim \ref{cl3}. Since $\|1_{F_i}\|/|1_A\| > \min F_i/3 \rightarrow \infty$ as $i\rightarrow\infty$, the basis $\mathcal{B}$ is not $\mathcal{S}_{\alpha+1}$-disjoint democratic. 
\end{proof}

By Claims \ref{cl2}, \ref{cl3}, and \ref{cl4}, our basis $\mathcal{B}$ is ($\infty$, $\alpha$)-quasi-greedy.

\section{Proof of Theorem \ref{m31}}

Before proceeding to the proof of Theorem \ref{m31}, we isolate the following simple lemma but omit its straightforward proof.

\begin{lem}\label{l40}
Let $\alpha < \omega_1$ and 
 $S$ be a finite set of positive integers with $\min S \geqslant 2$. Then there is an $m\in \mathbb{N}$ so that  $S \in \mathcal{S}_{\alpha + m}$.
\end{lem}


\begin{proof}[Proof of Theorem \ref{m31}]
Assume that our basis $(e_n)$ is greedy. Let $m\in\mathbb{N}$. By Konyagin and Temlyakov's characterization of greedy bases \cite{KT1}, we know that $(e_n)$ is $K$-unconditional and $\Delta$-democratic for some $K, \Delta \geqslant 1$. It follows from the definitions that $(e_n)$ is $K$-$\mathcal{S}_{\alpha + m}$-unconditional, $\Delta$-$\mathcal{S}_{\alpha + m}$-disjoint democratic, and $K$-quasi-greedy. By the proof of Proposition \ref{p20} and Theorem \ref{m1}, $(e_n)$ is $C$-$\mathcal{S}_{\alpha + m}$-greedy for some $C = C(K, \Delta)$.

Conversely, assume that $(e_n)$ is $C$-$\mathcal{S}_{\alpha + m}$-greedy for all $m\in\mathbb{N}$ and some uniform $C\geqslant 1$. We need to show that $(e_n)$ is unconditional and disjoint democratic. Let $A\subset\mathbb{N}$ be a finite set. Write $A = (A\cap \{1\})\cup (A\backslash \{1\})$. By Lemma \ref{l40}, there exists $m$ such that $A\backslash \{1\}\in \mathcal{S}_{\alpha + m}$. Hence, $\mathcal{S}_{\alpha + m}$-unconditionality implies that $\|P_{A\backslash \{1\}}\|\leqslant C+1$ (see Theorem \ref{m1}). Therefore, 
$$\|P_A\|\ \leqslant\ \|e_1^*\|\|e_1\| + C+1\ \leqslant\ c_2^2 + C + 1,$$
and so, $(e_n)$ is unconditional. Next, we show that $(e_n)$ is disjoint democratic. Pick finite disjoint sets $A, B\subset\mathbb{N}$ with $|A|\leqslant |B|$. Since $A\backslash \{1\}\in \mathcal{S}_{\alpha+m}$ for some sufficiently large $m$ and $(e_n)$ is $C$-$\mathcal{S}_{\alpha + m}$-disjoint democratic, $\|1_{A\backslash \{1\}}\|\leqslant C\|1_B\|$. Furthermore, 
$$\|1_{A\cap \{1\}}\|\ \leqslant\ c_2\ \leqslant\ c_2\sup_{n}\|e^*_n\|\|1_B\|\ \leqslant\ c_2^2\|1_B\|.$$
We obtain
$$\|1_A\|\ \leqslant\ (C+c_2^2)\|1_B\|.$$
Hence, $(e_n)$ is disjoint democratic. This completes our proof. 

Finally, we show that there exists a basis that is $\mathcal{S}_{\alpha+m}$-greedy for all $m\in \mathbb{N}$ but is not greedy. Let $\beta$ be the smallest limit ordinal that is greater than $\alpha + m$ for all $m\in \mathbb{N}$. Consider the canonical basis $(e_n)$ of the space $X_{\beta, \infty}$ in Subsection \ref{rt}. We have shown that $(e_n)$ is $\mathcal{S}_\beta$-greedy. By Corollary \ref{m30'}, $(e_n)$ is $\mathcal{S}_{\alpha + m}$-greedy for all $m$. However, since the basis is not unconditional, it is not greedy. 
\end{proof}

\section{Future research}
In this paper, we show that given a pair $(\alpha,\beta) \in (\omega_1\cup \{\infty\})^2$, if either $\alpha$ or $\beta$ is $\infty$ or if $(\alpha,\beta) = (0,0)$, there is a Banach space with an $(\alpha, \beta)$-quasi-greedy basis. The result is sufficient enough to prove Theorem \ref{m30}. 
A natural extension of our work is whether there is an $(\alpha,\beta)$-quasi-greedy basis for every pair $(\alpha,\beta) \in (\omega_1\cup \{\infty\})^2$. 

Regarding Theorem \ref{m31}, we would like to know whether an $\mathcal{S}_\alpha$-greedy basis for all countable ordinals $\alpha$ (with different greedy constants) is greedy. Similarly, must an $\mathcal{S}_\alpha$-unconditional basis for all countable ordinals $\alpha$ be unconditional?

\section{Appendix}

\begin{lem}\label{lems2} The following hold. 
\begin{enumerate}
\item [i)] If $F\in \mathcal{S}_\alpha$ for some $\alpha$ and $\min F = 1$, then $F = \{1\}$.
\item [ii)] For all ordinals $\alpha\geqslant 0$, $\mathcal{S}_0\subset \mathcal{S}_\alpha$.
\item [iii)] For all ordinals $\alpha\geqslant 2$, $\mathcal{S}_2\subset \mathcal{S}_\alpha$.
\end{enumerate}
\end{lem}

\noindent We omit the straightforward proof of Lemma \ref{lems2}. For completeness, we include the easy proof of the following lemma.

\begin{lem}\label{8times}
Fix $\alpha\geqslant 2$ and $M\in [\mathbb{N}]$, $\min M \geqslant 3$.
Let $\ell_n = \min \alpha(M, n)$. It holds that 
$\ell_{n+1}\geqslant 8\ell_n$ for all $n\geqslant 1$. 
\end{lem}

\begin{proof}
Let $L_n = M\backslash \cup_{i=1}^{n-1}\supp(\alpha(M, i))$ for $n\geqslant 1$. Then $\min L_n = \ell_n$ for all $n\geqslant 1$. First, we show that, 
\begin{equation}\label{etrou}\max\supp (\alpha(M,n))\ \geqslant\ \max\supp (2(L_n, 1)), \forall n\geqslant 1.\end{equation}
Suppose, for a contradiction, for some $n$, 
$$\max\supp (\alpha(M, n))\ <\ \max\supp (2(L_n, 1)).$$
Let $E = \supp (\alpha(M, n))$ and $F = \supp (2(L_n,1))$. Then $E\subsetneq F$. Since $F\in \mathcal{S}_2$, $F\in \mathcal{S}_{\alpha}$ according to Lemma \ref{lems2}. That $E\subsetneq F$ and $F\in \mathcal{S}_\alpha$ contradict that $E$ is a maximal $\mathcal{S}_\alpha$-set. Therefore, for all $n\geqslant 1$, \eqref{etrou} holds.

We have for all $n\geqslant 1$,
$$\frac{\ell_{n+1}}{\ell_n}\ \geqslant\ \frac{\max\supp (\alpha(M, n))+1}{\ell_n}\ \geqslant\ \frac{\max\supp (2(L_n, 1))+1}{\ell_n}\ \geqslant\ \frac{2^{\ell_n} \ell_n}{\ell_n}\ \geqslant \ 8.$$
This completes our proof. 
\end{proof}

\end{document}